\documentclass[a4paper, 10pt, twoside]{article}

\usepackage{exscale}
\usepackage[centertags]{amsmath}
\usepackage{amssymb}
\usepackage{amsthm}
\usepackage{dsfont}
\usepackage{color}
\usepackage{url}
\usepackage[hypertex]{hyperref}
\usepackage[curve,matrix,arrow,ps]{xy}

\newtheorem{definition}{Definition}

\newtheorem{theorem}[definition]{Theorem}
\newtheorem{proposition}[definition]{Proposition}
\newtheorem{lemma}[definition]{Lemma}

\newtheorem{deflemma}[definition]{Definition-Lemma}

\newcommand{\nd}{\noindent}

\newcommand{\dC}{{\mathds C}}

\newcommand{\dN}{{\mathds N}}

\newcommand{\dZ}{{\mathds Z}}
\newcommand{\dP}{{\mathds P}}

\newcommand{\dH}{{\mathbb H}}

\newcommand{\bD}{{\mathbb D}}

\newcommand{\cD}{\mathcal{D}}
\newcommand{\cE}{\mathcal{E}}
\newcommand{\cF}{\mathcal{F}}
\newcommand{\cG}{\mathcal{G}}
\newcommand{\cH}{\mathcal{H}}

\newcommand{\cK}{\mathcal{K}}

\newcommand{\cM}{\mathcal{M}}

\newcommand{\cO}{\mathcal{O}}

\newcommand{\cR}{\mathcal{R}}

\newcommand{\cT}{\mathcal{T}}

\newcommand{\D}{\displaystyle}

\newcommand{\SC}{\scriptstyle}

\DeclareMathOperator{\Spec}{\textup{Spec}\,}

\DeclareMathOperator{\diag}{\textup{diag}}
\DeclareMathOperator{\vol}{\textup{vol}}
\DeclareMathOperator{\gr}{\textup{gr}}

\DeclareMathOperator{\Der}{Der}

\DeclareMathOperator{\Gl}{Gl}

\begin{document}

\title{Duality of Gau\ss-Manin systems associated to linear free divisors}
\author{Christian Sevenheck}

\maketitle

\begin{abstract}
We investigate differential systems occurring in the
study of particular non-isolated singularities, the so-called
linear free divisors. We obtain a duality theorem for these
$\cD$-modules taking into account filtrations, and deduce
degeneration properties of certain Frobenius manifolds associated to
linear sections of the Milnor fibres of the divisor.
\end{abstract}
\renewcommand{\thefootnote}{}
\footnote{
This research is supported by a DFG Heisenberg fellowship (Se 1114/2-1) and by ANR grant ANR-08-BLAN-0317-01 (SEDIGA).
}

\section{Introduction}
\label{sec:Introduction}

The aim of the present note is to study the duality theory of some
particular differential systems, which were introduced in \cite{dGMS}.
These are Gau\ss-Manin systems associated to hyperplane sections of
Milnor fibres of special non-isolated singularities, called linear
free divisors. These Gau\ss-Manin systems are used in loc.cit. for the construction of families of
Frobenius manifolds (following the general framework from \cite{DS}), which generalizes the Frobenius structure defined by
the quantum cohomology of the projective space. A natural question one may ask in this context
is whether they also arise as the quantum cohomology of some variety
(or orbifold). A particular and interesting class of examples
of linear free divisors are discriminants in representation spaces of a quiver $Q$ (see, e.g.,
\cite{BM} and \cite{GMNS}).
In that case, one may also ask about the relationship between the
Frobenius structures constructed in loc.cit. and the conjectured
Frobenius manifold structure on the space of stability conditions in $D^b(mod\,kQ)$ (see ,e.g.,  \cite{BridgelandSurvey} or \cite{Takahashi1}).
In any case, a precise study of the various properties
of linear free divisors and of the Frobenius manifolds from \cite{dGMS}
is of interest in both singularity theory and representation theory.
Particularly important is a detailed understanding
of the degeneration behavior at the limit point of the parameter space
(corresponding to the hyperplane section of $D$ itself, and generalizing the
large radius limit of the quantum cohomology of $\dP^n$).
Some questions on the limit behavior of the Frobenius manifolds constructed in \cite{dGMS}
remained open in that paper because of a lack of understanding of the duality theory
of the Gau\ss-Manin systems associated to the above mentioned hyperplane sections. In this paper,
we prove a conjecture from loc.cit. and give some consequences on this degeneration behavior.
The basic tool for this proof is an explicit description
of the Gau\ss-Manin system by differential operators, for which the duality statements
needed can be calculated directly.

Let us notice that the mirror of the quantum cohomology of the projective spaces can also be generalized
by the mirror Landau-Ginzburg model of a weighted projective space, in that case, one can similarly study
the associated filtered Gau\ss-Manin systems, and due to the more explicit control of the relevant cohomological invariants
(the spectral numbers, see proposition \ref{prop:GoodBasis} below), the corresponding statements have been shown in \cite{DM}, following
general results on a purely algebraic construction of Frobenius manifolds in \cite{Dou2}. The main
point in this note is to obtain these results for the Gau\ss-Manin systems of linear sections of Milnor fibres of linear free divisors,
where due to the more complicated combinatorial structure of the input data (like quiver representations)
the distribution of the spectral numbers is less easy to control.


\textbf{Acknowledgements: } I would like to thank Thomas Reichelt and
John Alexander Cruz Morales for interesting discussions on topics related to this article.

\section{Linear free divisors, hyperplane sections and Gau\ss-Manin systems}
\label{sec:LFDandCompact}

We start by introducing the main objects of interest of this paper.
We also give a short but self-contained account of the
results from \cite{dGMS} needed here. More details can be found in loc.cit.
and in \cite{Sev1}.
\begin{deflemma}
\label{def:LFD}
Write $V$ for the affine space $\dC^n$ with coordinates $x_1,\ldots,x_n$. Let $D\subset V$ be a reduced hypersurface,
given by a polynomial equation $h\in\cO_V$.
\begin{enumerate}
\item
$D$ is called a linear free divisor iff there is a basis $\vartheta_1,\ldots,\vartheta_n$ of
the $\cO_V$-module $\Theta(-\log\,D):=\{\vartheta\in\Theta_V\,|\,\vartheta(h)\subset(h)\}$
(in particular, $\Theta(-\log\,D)$ must be $\cO_V$-free) such that
$\vartheta_i = \sum_{i=1}^n a_{ij}\partial_{x_j}$ where  $a_{ij} \in \dC[V]_1$ is a linear form.
We will also consider the submodule $\Theta(-\log\,h)=\{\vartheta\in\Theta(-\log\,D)\,|\,\vartheta(h)=0\}$.
$\Theta(-\log\,h)$ is $\cO_V$-free of rank $n-1$ and we have
$\Theta(-\log\,D) = \Theta(-\log\,h)\oplus\cO_V\cdot E$, where $E=\sum_{i=1}^n x_i\partial_{x_i}$.
\item
$D$ is called reductive, if the identity component $G$ of the algebraic group
$G_D:=\{g\in\Gl(V)\,|\,g(D)\subset D\}$ is so.
\item
For any reductive linear free divisor $D$, the dual action of $G$ on $V^\vee$
has an open orbit, with complement a reduced hypersurface $D^\vee$. Any linear
form $f$ in the open orbit $V^\vee\backslash D^\vee$ is called generic.
\end{enumerate}
\end{deflemma}
We are interested in Gau\ss-Manin systems of the restriction of generic linear forms
to the Milnor fibres of $D$, this leads to consider relative de Rham complexes with
twisted differentials.
\begin{deflemma}
Let $D\subset V$ be a reductive linear free divisor with defining equation $h$, seen as a morphism
$h:V\rightarrow T=\Spec\dC[t]$. Moreover, let $f\in V^\vee\backslash D^\vee$ be
a generic linear form, which we see as a morphism $f:V\rightarrow S=\Spec\dC[s]$.

Put
$$
\begin{array}{c}
G(\log\,D):=\dH^n(\Omega^\bullet_{V/T}(\log\,D)[\theta,\theta^{-1}],\theta d - df\wedge)
\quad;\quad
G(*D):=\dH^{n-1}(\Omega^\bullet_{V/T}(*D)[\theta,\theta^{-1}],\theta d - df\wedge)
\\ \\
G_0(\log\,D):=\dH^n(\Omega^\bullet_{V/T}(\log\,D)[\theta],\theta d - df\wedge)
\quad;\quad
G_0(*D):=\dH^{n-1}(\Omega^\bullet_{V/T}(*D)[\theta],\theta d - df\wedge)
\end{array}
$$
where
$\Omega^\bullet_{V/T}(*D) :=\Omega^\bullet_V(*D)/h^*\Omega^1_T(*\{0\})\wedge\Omega^{\bullet-1}_V(*D)$
resp.
$\Omega^\bullet_{V/T}(\log\,D) :=\Omega^\bullet_V(\log\,D)/h^*\Omega^1_T(\log\,\{0\})\wedge\Omega^{\bullet-1}_V(\log\,D)$
is the localization along $D$ resp. its logarithmic extension over $D$
of the relative de Rham complex of $h$.
Then $G(*D)$ is $\dC[\theta,\theta^{-1},t,t^{-1}]$-free of rank $n$ and
$G(\log\,D)$ (resp. $G_0(*D), G_0(\log\, D)$) is a
$\dC[\theta,\theta^{-1},t]$- (resp. $\dC[\theta,t,t^{-1}]$-,
$\dC[\theta,t]$-) lattice inside $G(*D)$.
There is a connection operator $$ \nabla:G_0(\log\,D) \longrightarrow G_0(\log\,D) \otimes z^{-1}\Omega^1_{\dC\times T}\left(\log((\{0\}\times T)\cup(\dC\times\{0\}))\right) $$ which induces connections on $G(\log\, D)$, $G_0(*D)$ and $G(*D$).
\end{deflemma}

We have constructed in \cite{dGMS} two particular bases of $G_0(*D)$ in which the connection can be expressed
in a very simple way. This is summarized in the following proposition.
\begin{proposition}\label{prop:GoodBasis}
\begin{enumerate}
\item
There is a $\dC[\theta,t,t^{-1}]$-basis $\underline{\omega}=(\omega_1,\ldots,\omega_n)$ (called $\underline{\omega}^{(2)}$ in \cite[corollary 4.12]{dGMS}) of $G_0(*D)$ such
that
\begin{equation}\label{eq:Basis}
\nabla(\underline{\omega}) = \underline{\omega} \cdot
\left[
(A_0\frac{1}{\theta}+A_\infty)\frac{d\theta}{\theta}+(-A_0\frac{1}{\theta}+A'_\infty)\frac{dt}{nt}
\right]
\end{equation}
where
$$
A_0:=
  \begin{pmatrix}
    0 & 0&\ldots &0& c\cdot t\\
    -1   &0&\ldots&0& 0\\
    \vdots & \vdots & \ddots & \vdots &\vdots \\
    0 &0 & \ldots &0& 0\\
    0&0& \ldots& -1&0
  \end{pmatrix},
$$
$c\in\dC^*$, $A_\infty=\diag(\nu_1,\ldots,\nu_n)$ and
$A'_\infty:=\diag(0,1,\ldots,n-1)-A_\infty$. The numbers $\nu_i$ have the following two properties:
\begin{enumerate}
\item
For all $i\in\{1,\ldots,n-1\}$, we have $\nu_{i+1}-\nu_i\leq 1$.
\item
Let $\sigma\in S_n$ be a permutation
such that $\nu_{\sigma(1)} \leq \ldots \leq \nu_{\sigma(n)}$.
Then we have the symmetry $\nu_{\sigma(i)}+\nu_{\sigma(n+1-i)}=n-1$.
\end{enumerate}
We have $G_0(\log\,D)=\oplus_{i=1}^n \dC[t,\theta]\omega_i$ and the basis $\underline{\omega}$ is a $V^+$-solution
to the Birkhoff problem (see \cite[appendix B.d]{DS}) of the module $(G_0(\log\,D)/t\cdot G_0(\log\,D),\nabla)$,
which was called logarithmic Brieskorn lattice in \cite{Sev1}.
\item
There is another basis $\widetilde{\underline{\omega}}$ of $G_0(*D)$
(called $\underline{\omega}^{(3)}$ in \cite[corollary 4.12]{dGMS}) for which the connection matrix takes the
same form as for $\underline{\omega}$, that is
$$
\nabla(\widetilde{\underline{\omega}}) = \widetilde{\underline{\omega}} \cdot
\left[
(A_0\frac{1}{\theta}+\widetilde{A}_\infty)\frac{d\theta}{\theta}+(-A_0\frac{1}{\theta}+\widetilde{A}'_\infty)\frac{dt}{nt}
\right]
$$
where $\widetilde{A}_\infty=\diag(\widetilde{\nu}_1,\ldots,\widetilde{\nu}_n)$ and $\widetilde{A}'_\infty=\diag(0,1,\ldots,n-1)-\widetilde{A}_\infty$.
Here the numbers $\widetilde{\nu}_i$ have the same properties as the numbers $\nu_i$ above and satisfy additionally
$\widetilde{\nu}_1-\widetilde{\nu}_n\leq 1$, moreover, we have $\underline{\omega}=\widetilde{\underline{\omega}}$ if $\nu_1-\nu_n\leq 1$. $\widetilde{\underline{\omega}}$ is a $V^+$-solution to the Birkhoff problem in a family for $(G_0(*D),\nabla)$.
\end{enumerate}
\end{proposition}
The following theorem is a rather easy consequence of this result, but it will be very useful when studying the duality theory
of $G(*D)$.
\begin{theorem}
\label{theo:GMSystemCyclic}
Write $\cD=\dC[\theta,\theta^{-1},t,t^{-1}]\langle \partial_t, \partial_\theta\rangle$. Then there
is an isomorphism of left $\cD$-modules
\begin{equation}\label{eq:Ident-G}
\begin{array}{rcl}
\varphi:\cG:=\cD/\cD\cdot (P_1,P_2) & \longrightarrow & G(*D) \\ \\
1 & \longmapsto & \left[ n\vol/dh\right],
\end{array}
\end{equation}
where
$$
P_1  :=  \prod_{i=1}^n \theta\left(t\partial_t-\frac{i-1-\nu_i}{n}\right) - \frac{c}{n^n}\cdot t\\
\quad;\quad
P_2  :=  \theta^2\partial_\theta + n \cdot t\theta\partial_t
$$
and where $\vol=dx_1\wedge\ldots\wedge dx_n$.
The inverse image of the lattice $G_0(*D)$ under $\varphi$ can be described as the subring $\cG_0$ of $\cD/\cD\cdot (P_1,P_2)$ defined as
$$
\cG_0:=\dC[\theta,t,t^{-1}]\langle \theta\partial_t, \theta^2\partial_\theta \rangle \left/\dC[\theta,t,t^{-1}]\langle \theta\partial_t, \theta^2\partial_\theta \rangle\cdot (P_1, P_2).\right.
$$
Here $\dC[\theta,t,t^{-1}]\langle \theta\partial_t, \theta^2\partial_\theta \rangle$ is
the $\dC[\theta,t,t^{-1}]$-subalgebra of $\cD$ generated by $\theta\partial_t$ and $\theta^2\partial_\theta$.

\end{theorem}
\begin{proof}
By forgetting the $\partial_\theta$-action, we can see $G(*D)$ as a
$\dC[\theta,\theta^{-1},t,t^{-1}]\langle \partial_t \rangle$-module only. The first step
is then to show that there is an isomorphism of $\dC[\theta,\theta^{-1},t,t^{-1}]\langle \partial_t \rangle$-modules
$$
\begin{array}{rcl}
\widetilde{\varphi}:\widetilde{\cG}:=\dC[\theta,\theta^{-1},t,t^{-1}]\langle \partial_t \rangle \left/\dC[\theta,\theta^{-1},t,t^{-1}]\langle \partial_t \rangle \cdot P_1\right. & \longrightarrow & G(*D) \\ \\
1 & \longmapsto & \left[ n\vol/dh\right],
\end{array}
$$
From the form of the operator $P_1$ we see that the left hand side is generated over $\dC[\theta,\theta^{-1},t,t^{-1}]$ by
the elements $\left(Q_i:=\prod_{j=1}^i \theta(t\partial_t-\frac{j-1-\nu_j}{n})\right)_{i=0,\ldots,n-1}$, where we put by definition $Q_0:=1$.
Hence the operator $\theta t\partial_t-\frac{i-1-\nu_i}{n}$ sends $Q_{i-1}$ to $Q_i$, so that by
setting  $\widetilde{\varphi}(Q_i):=n^{-i}\omega_{i+1}$, we obtain a $\dC[\theta,\theta^{-1},t,t^{-1}]$-linear map which is compatible
with the action of $\partial_t$ on the left hand side and $\nabla_{\partial_t}$ on the right hand side.
Due to the particular form of the operator $P_2$ (more precisely, due to the fact that $\theta^{-2}\cdot P_2
=\partial_\theta+n\theta^{-1}t\partial_t\in(P_1,P_2)$), the module $\cG$ is isomorphic to $\widetilde{\cG}$ when seen as a $\dC[\theta,\theta^{-1},t,t^{-1}]\langle \partial_t \rangle$-module only. Hence in order to finish the proof of the first statement, we have to check that the action of $\partial_\theta$ on $\cG$
coincides with the action of $\nabla_{\partial_\theta}$ on $G(*D)$, which is clear from formula \eqref{eq:Basis},
and by noticing that
$$
\begin{array}{rcl}
(\theta^2\partial_\theta -\theta\nu_{i-1}) \cdot Q_{i-1} &= &
(i-1) \theta^i \cdot R_{i-1} - \theta^{i-1}\cdot R_{i-1}\cdot n\theta t\partial_t-\theta^i \nu_i\cdot R_{i-1} \\
\theta^i\cdot \left(    -nt\partial_t+i-1-\nu_i\right)\cdot R_{i-1} & = & -n \cdot Q_i,
\end{array}
$$
where we write $R_i=\prod_{j=1}^i (t\partial_t-\frac{j-1-\nu_j}{n})$.

Looking at the connection matrix \eqref{eq:Basis}, one immediately sees that
$\varphi(\cG_0)\subset G_0(*D)$. In order to show equality, take any section $[\omega]\in G_0(*D)$ with representative
$\omega=\sum_{k\geq 0} \theta^k \omega^{(k)}$, where $\omega^{(k)}\in\Omega^{n-1}_{V/T}(*D)$. There is an (uniquely determined)
operator $P$ in $\cG$ with $\varphi(P)=[\omega]$, and we have to show that $P\in\cG_0$. By linearity of $\varphi$,
it is sufficient to do it for the case where $\omega^{(0)}\neq 0$, and then there is a minimal $k\in\dN$ with
$\theta^k\cdot P\in\cG_0$, and then the class of $\theta^k\cdot P$ in $\cG_0/\theta\cG_0$ is non-zero.
On the other hand, $\theta^k\varphi(P)=\theta^k\omega\in G_0(*D)$, and the class of $\theta^k\omega$ is zero
in $G_0(*D)/\theta \cdot G_0(*D)$ unless $k=0$. Hence the statement follows once we know that
the induced morphism
$$
[\widetilde{\varphi}]: \cG_0/\theta \cG_0 \longrightarrow G_0(*D)/\theta G_0(*D)
$$
is an isomorphism.
Now recall from \cite[section 3.2]{dGMS} that the relative deformation or Jacobian
algebra is defined as
$$
\cT^1_{\cR_h/T}(f):= \frac{\cO_V}{df(\Theta(-\log\,h))}.
$$
Then we have $G_0(*D)/\theta G_0(*D) \cong H^0(V,\cT^1_{\cR_h/T}(f))\cdot \vol$.
On the other hand, $\cG_0/\theta \cG_0 \cong
\widetilde{\cG}_0/\theta \widetilde{\cG}_0 = \dC[t,t^{-1},\mu]/((t\mu)^n-c/n^n\cdot t)$, where
we denote by
$\mu$ the class of $[\theta\partial_t]$ in $\widetilde{\cG}_0/\theta \widetilde{\cG}_0$.
Then the isomorphism $\dC[t,t^{-1},\mu]/((t\mu)^n-c/n^n\cdot t)\cong H^0(V,\cT^1_{\cR_h/T}(f))$ follows
from \cite[proposition 3.5]{dGMS}.
%
%
\end{proof}

\section{Duality theorems}
\label{sec:Duality}

In this sections we derive the existence of a pairing
on the meromorphic bundle $G(*D)$. For that purpose, we compute
the holonomic dual of the module $\cG$. We
show that it is self-dual, by exhibiting a $\cD$-free
resolution of it. Using theorem \ref{theo:GMSystemCyclic} and
a comparison result between the meromorphic and the holonomic dual module
of a meromorphic bundle, this yields the pairing mentioned above.
%
We also show that it is compatible with the lattice $G_0(*D)$, this fact
is used later for the construction of the flat metric on the Frobenius manifold associated to the pair $(f,h)$.

\begin{proposition}
\label{prop:HolDual-Alg}
\begin{enumerate}
\item
Let $\cG^r:=\cD/(P^t_1,-P_2^t)\cdot\cD$ be the right module associated to $\cG$, here
$$
P^t_1  :=  \prod_{i=1}^n (-\theta)\left(t\partial_t+\frac{i-1-\nu_i}{n}+1\right) - \frac{c}{n^n}\cdot t\\
\quad;\quad
P^t_2  :=  -\left(\theta^2\partial_\theta + n \cdot t\theta\partial_t + \theta(n+2)\right).
$$
are the usual transforms of $P_1$ and $P_2$.
Then $\cG^r$ has a the following explicit resolution by free right $\cD$-modules
\begin{equation}
\label{eq:Resolution}
\xymatrix@!0{
0 \ar[rr] && \cD \ar[rrr]^{\begin{pmatrix}\SC P^t_2 - n\theta \\ \SC P^t_1\end{pmatrix}\cdot} &&& \cD^2 \ar[rrr]^{\begin{pmatrix}\SC P^t_1 &\SC -P^t_2\end{pmatrix}\cdot} &&& \cD \ar[rrr] &&& \cG^r \ar[rr] &&0
}
\end{equation}
\item
We have $\bD\cG=\cD/\cD\cdot(\widetilde{P}_1,\widetilde{P}_2)$ where
$\widetilde{P}_1=P^t_1$ and $\widetilde{P}_2:=-P^t_2+n\theta=\theta^2\partial_\theta+nt\theta\partial_t+\theta(2n+2)$.
Moreover, there is an isomorphism $\Phi:\cG\rightarrow \iota^*\bD\cG$ of left $\cD$-modules, induced by
$\Phi(m)= m\cdot\theta^{n+2} \cdot t$ for any $m\in \cG$, where $\iota$ is the involution sending $z$ to $-z$.
\end{enumerate}
\end{proposition}
\begin{proof}
\begin{enumerate}
\item
From the commutation relation
$[P^t_1, P^t_2] = n\cdot\theta\cdot P^t_1$ we conclude that \eqref{eq:Resolution} is a complex. In order to show that it is exact,
filter $\cD$ as usual by orders of operators. Then it suffices 
to show that the graded object with respect to this filtration of the complex \eqref{eq:Resolution} is acyclic. It is easy to see that the symbols
$\sigma(P^t_1)=(-\theta t \sigma(\partial_t))^n$ and $\sigma(P^t_2)=-\theta^2\sigma(\partial_\theta)-nt\theta\sigma(\partial_t)$ form a regular sequence
in $\gr(\cD)$, so that the only relation between them
is the Koszul relation, and the corresponding (exact) Koszul complex is exactly the graded complex associated
to the above complex of $\cD$-modules, which is hence acyclic.
\item
As the above sequence \eqref{eq:Resolution} is a free right resolution, we obtain that
$\bD\cG$ is the top cohomology group of its dual complex (which is naturally a complex of left $\cD$-modules,
notice that we could have as well started with a left resolution of $\cG$, use it to compute the right
$\cD$-module $\mathit{Ext}^2_\cD(\cG,\cD)$ and then obtain $\bD\cG$ as the left transform of it).
It follows from \cite[corollary 17]{Sev1} that
$\widetilde{P}_1 = \prod_{i=1}^n (-\theta)\left(t\partial_t-\frac{i-1-\nu_i}{n}+1\right) - \frac{c}{n^n}\cdot t$,
due to the symmetry of the numbers $\frac{i-1-\nu_i}{n}$ deduced in loc.cit. from the symmetry of the roots of the Bernstein polynomial $b_h$
shown in \cite{GS}. Now one checks explicitly that the morphism
$$
\begin{array}{rcl}
\Phi:\cD/\cD\cdot (P_1,P_2) & \longrightarrow & \iota^*\left(\cD/\cD\cdot (\widetilde{P}_1,\widetilde{P}_2)\right) \\ \\
m & \longmapsto & m\cdot \theta^{n+2} \cdot t
\end{array}
$$
is well-defined, i.e., that $P_1(\theta^{n+2}t)=P_2(\theta^{n+2}t)=0\in \iota^*\left(\cD/\cD(\widetilde{P}_1,\widetilde{P}_2)\right)$.
It is obviously invertible, and hence yields the desired
isomorphism $\Phi:\cG\rightarrow \iota^*\bD\cG$.
\end{enumerate}

\end{proof}

From the above calculation we can now deduce the first main result.
\begin{theorem}\label{theo:Pairing}
There is a non-degenerate, $(-1)^{n-1}$-symmetric pairing
$S:G(*D) \otimes \iota^*G(*D) \rightarrow \dC[\theta,\theta^{-1},t,t^{-1}]$, which is compatible with
the connections.
\end{theorem}
\begin{proof}
We first recall a construction from \cite[lemma A.11]{DS} and \cite[section 2.7]{SM}
which yields a canonical resolution of the right $\cD$-module associated to $G(*D)$ and, as a consequence,  an identification of the holonomic and the meromorphic dual of $G(*D)$. We will
write $\cO:=\cO_{\widehat{S}^*\times T^*}$ and $\Omega^i:=\Omega^i_{\widehat{S}^*\times T^*}$, with $\widehat{S}^*=\Spec\dC[\theta,\theta^{-1}]$.

Consider the de Rham complex $\Omega^\bullet(\cD)$ of $\cD$ which is a resolution by free right $\cD$-modules of
$\Omega^2$, i.e., the exact sequence
$$
0 \longrightarrow \cD  \stackrel{\alpha'}{\longrightarrow} \Omega^1\otimes_{\cO} \cD
\stackrel{\beta'}{\longrightarrow} \Omega^2\otimes_{\cO} \cD \stackrel{\gamma'}{\longrightarrow} \Omega^2 \longrightarrow 0,
$$
where $\alpha'(P)=d\theta\otimes (\partial_\theta\cdot P)+dt\otimes (\partial_t\cdot P)$,
$\beta'(d\theta \otimes P_1 + dt\otimes P_2)=(d\theta\wedge dt)\otimes (\partial_\theta\cdot P_2 - \partial_t \cdot P_1)$
and $\gamma'((d \theta \wedge dt)\otimes Q)=(d \theta \wedge dt)\cdot Q$, where the last term denotes the result of the
right action of the operator $Q$ on the element $d \theta \wedge dt\in \Omega^2$.

Now recall that there is a sequence of isomorphism of right $\cD$-modules
\begin{equation}\label{eq:SaitoIso}
\begin{array}{rcl}
\left(\Omega^i\otimes \cD\right)\otimes G(*D) \cong \Omega^i\otimes\left(\cD\otimes G(*D)\right)
\cong \Omega^i \otimes \left(G(*D)\otimes\cD\right)\cong \left(\Omega^i \otimes G(*D)\right)\otimes\cD
\end{array}
\end{equation}
where all tensor products are over $\cO$, where the left-most and the right-most isomorphisms are the obvious ones, and where the middle-one is induced by
the isomorphism $\cD\otimes G(*D)\stackrel{\cong}{\rightarrow} G(*D)\otimes \cD$ sending $P\otimes m$ to $P\cdot(m\otimes 1)$. Notice
that here $G(*D)\otimes \cD$ carries the trivial right $\cD$-module structure (i.e., the one coming from right multiplication on  the second
factor), but also the left $\cD$-module structure induced by the action $g(m\otimes P)=m\otimes gP$ for $g\in\cO$ and $\xi(m\otimes P):=\xi m\otimes P + m\otimes \xi P$ for $\xi\in\Der(\cO,\cO)\subset \cD$. Similarly, $\cD\otimes G(*D)$ has the trivial left structure, but also a (non-trivial) right structure
defined similarly to the left structure of $G(*D)\otimes \cD$.
Using these isomorphisms, one checks that the right $\cD$-module complex
$\Omega^\bullet(\cD)\otimes G(*D)$ is isomorphic to
\begin{equation}\label{eq:SaitoRes}
0\longrightarrow G(*D)\otimes_\cO \cD \stackrel{\alpha}{\longrightarrow} \Omega^1\otimes_\cO G(*D)\otimes_\cO \cD \stackrel{\beta}{\longrightarrow} \Omega^2\otimes_\cO G(*D)\otimes_\cO \cD
\stackrel{\gamma}{\longrightarrow} \Omega^2\otimes_\cO G(*D)\longrightarrow 0
\end{equation}
where $\alpha$ (resp. $\beta$ and $\gamma$) are induced from $\alpha'\otimes \mathit{Id}_{G(*D)}$ (resp.
$\beta'\otimes \mathit{Id}_{G(*D)}$  and $\gamma'\otimes \mathit{Id}_{G(*D)}$) under the isomorphisms \eqref{eq:SaitoIso} and can be expressed
explicitly as follows
\begin{eqnarray*}
\alpha(m \otimes P) & = & d\theta \otimes \left(\nabla_\theta m\otimes P + m\otimes\partial_\theta P \right)+dt\otimes\left(\nabla_t m\otimes P+m\otimes\partial_t P\right)\\ \\
\beta(d\theta \otimes m_1\otimes P_1 + dt \otimes m_2\otimes P_2) & = & d\theta\wedge dt \otimes \left(\nabla_\theta m_2\otimes P_2+m_2\otimes\partial_\theta P_2
-\nabla_t m_1 \otimes P_1-m_1\otimes\partial_t P_1\right)\\ \\
\gamma(d\theta\wedge dt \otimes (m\otimes P)) & = & ((d\theta\wedge dt)\cdot P)\otimes m - (d\theta\wedge dt)\otimes (P\cdot m)
\end{eqnarray*}
where, as before $(d\theta\wedge dt)\cdot P$ denotes the right action of the operator $P\in\cD$ on the form $d\theta\wedge dt\in\Omega^2$ and
$P\cdot m$ denotes the left action of $P$ on $m$ using the connection on the meromorphic bundle $G(*D)$.
Notice that as $G(*D)$ is $\cO$-free, the complex \eqref{eq:SaitoRes} is still exact, in other words, it yields a canonical resolution by free right $\cD$-modules of the right module associated to $G(*D)$.

The isomorphism $\varphi:\cG\rightarrow G(*D)$ from theorem \ref{theo:GMSystemCyclic} induces an isomorphism $\varphi^r$ on
the associated right $\cD$-modules, and the latter
extends to an isomorphism of complexes
{\small
\begin{equation}\label{diag:HolMerDual}
\xymatrix{
\cD \ar@^{(->}[rr]^{\left(\begin{array}{c}P^t_2-n\theta\\P^t_1\end{array}\right)\cdot} \ar[dd]_{\psi}  && \cD^2 \ar[rr]^{\D\left(P^t_1\;\;-P^t_2\right)\cdot} \ar[dd] && \cD \ar@{->>}[r] \ar[dd]& \cG^r=\cD/(P^t_1,P^t_2)\cD \ar[dd]_{\varphi^r}^{\cong} \\ \\
G(*D)\otimes_\cO \cD \ar@^{(->}[rr]^\alpha         && \Omega^1\otimes_\cO G(*D)\otimes_\cO \cD \ar[rr]^\beta         && \Omega^2\otimes_\cO G(*D)\otimes_\cO \cD \ar@{->>}[r] & \Omega^2\otimes_\cO G(*D).
}
\end{equation}
}

Applying the functor ${\cH\!}om_\cD(-,\cD)$ to the
free part of the above diagram (i.e., to the morphism between the free resolutions of $\cG^r$ resp. $\Omega^2\otimes G(*D)$)
and using the isomorphism ${\cH\!}om_{\cD}(\cF\otimes \cD,\cD)={\cH\!}om_\cO(\cF,\cD) \cong \cD\otimes\cF^\vee$ for any $\cO$-free module $\cF$, we
see that the free resolution of $\Omega^2\otimes G(*D)$ is transformed to the complex $G(*D)^\vee\otimes\mathit{Sp}^\bullet(\cD)$, where $\mathit{Sp}^\bullet(\cD)$
denotes the Spencer complex of $\cD$, i.e., a resolution of $\cD$ by free left $\cD$-modules. This is a free left resolution
of the left $\cD$-modules $G(*D)^\vee$. Hence the transpose of $\psi$ induces
an isomorphism $\Psi:G(*D)^\vee \rightarrow \bD\cG$ of left $\cD$-modules.


The existence of the pairing $S$ can be rephrased as an isomorphism
$\Phi^{\mathit{mer}}:\left(G(*D),\nabla\right)\stackrel{\cong}{\rightarrow}\iota^* \left(G(*D),\nabla\right)^\vee$ of meromorphic bundles with connection (here
$\left(G(*D),\nabla\right)^\vee$ denotes the dual vector bundle together with its dual connection). We define
$\Phi^{mer}$ by the commutative diagram
{\small
\begin{equation}
\label{diag:PhiInduced}
\xymatrix{
\iota^*\bD \cG
 && \iota^*(G(*D),\nabla)^\vee \ar[ll]_{\iota^*\Psi}\\ \\
\cG \ar[rr]^\varphi \ar[uu]^\Phi && (G(*D),\nabla) \ar[uu]_{\Phi^{\mathit{mer}}}
}
\end{equation}
}
where $\Phi$ is the morphism from proposition \ref{prop:HolDual-Alg}, 2. Notice that
$\Phi^{mer}$ is an isomorphism as $\Phi$, $\varphi$ and $\iota^*\Psi$ are so.

In order to show the $(-1)^{n-1}$-symmetry of $S$, we use a variant of a criterion from
\cite[corollary 1.23]{DS}. Namely, it is sufficient to show that the morphism
$$
\iota^*\bD\Phi:\iota^*\bD(\iota^*\bD\cG)=\cG\longrightarrow \iota^*\bD \cG
$$
satisfies $\iota^*\bD\Phi = (-1)^{n-1}\cdot\Phi$. This can be proved by computing a resolution of the right module
$\iota^*(\bD\cG)^r=\iota^*{\cE\!}xt_\cD(\cG,\cD)$, extending the morphism $\Phi^r:\cG^r\rightarrow(\iota^*\bD\cG)^r$
to a morphism of the corresponding resolutions and dualizing. In other words, we consider the following morphism of complexes
\begin{equation}
\label{diag:DualizingHolDMap}
\xymatrix{
\cD \ar@^{(->}[rrr]^{\begin{pmatrix}\SC P_2 \\ \SC P'_1\end{pmatrix}\cdot} &&& \cD^2 \ar[rrrr]^{\begin{pmatrix}\SC P'_1 &\SC n\theta-P_2\end{pmatrix}\cdot} &&&& \cD \ar@{->>}[rr] && \left(\iota^*\bD\cG\right)^r
\\ \\
\cD \ar[uu]^{(-a)\cdot}\ar@^{(->}[rrr]^{\begin{pmatrix}\SC P^t_2-n\theta\\ \SC P^t_1\end{pmatrix}\cdot} &&& \cD^2
\ar[uu]_{\begin{pmatrix} \SC a & \SC 0 \\ \SC 0 & \SC -a\end{pmatrix}\cdot}
\ar[rrrr]^{\begin{pmatrix}\SC P^t_1 &\SC -P^t_2\end{pmatrix}\cdot} &&&& \cD \ar[uu]^{a \cdot} \ar@{->>}[rr] && \cG^r \ar[uu]^{\Phi^r}
}
\end{equation}
where $P'_1=\prod_{i=1}^n (-\theta)\left(t\partial_t-\frac{i-1-\nu_i}{n}\right) - \frac{c}{n^n}\cdot t$ and where we have put
$a:=\theta^{n+2}t$ for short. The dual of the leftmost morphism
induces the map
$$
\begin{array}{rcl}
\bD\Phi:\bD\iota^*\bD \cG \cong \iota^*\cG\cong \cD/\cD(P'_1,P_2) & \longrightarrow & \bD\cG=\cD/\cD(P^t_1,n\theta-P^t_2) \\ \\
m & \longmapsto & m \cdot (-a)=m\cdot (-1) \cdot \theta^{n+2}t
\end{array}
$$
Therefore the morphism $\iota^*\bD\Phi$ is given by right multiplication with $(-1)\cdot (-\theta)^{n+2}t$ and hence satisfies
$\iota^*\bD\Phi = (-1)^{n+1}\cdot\Phi=(-1)^{n-1}\cdot\Phi$, as required.

\end{proof}

In the remainder of this section, we will show a more refined version of theorem \ref{theo:Pairing} taking into account the behavior of the pairing $S$ with respect to the lattice $G_0(*D)$. More precisely, we have the following result.
\begin{theorem}\label{theo:PairingLattice}
The pairing $S$ from theorem \ref{theo:Pairing} satisfies
$S(G_0(*D),G_0(*D))\subset \theta^{n-1}\dC[\theta,\theta^{-1},t,t^{-1}]$. Moreover,
it induces a non-degenerate symmetric pairing
$S_0:(G_0(*D)/\theta\cdot G_0(*D))\otimes(G_0(*D)/\theta\cdot G_0(*D)) \rightarrow \theta^{n-1}\dC[t,t^{-1}]$.
\end{theorem}
\begin{proof}
It is clear that the statement of the theorem is equivalent to the fact that the morphism
$\Phi^{\mathit{mer}}:G(*D)\stackrel{\cong}{\longrightarrow}\iota^* G(*D)$ appearing in the proof of the previous
theorem sends $G_0(*D)$ isomorphically onto $\iota^*\left(\theta^{n-1}\cdot G_0(*D)^\vee\right)$, where
$$
G_0(*D)^\vee={\cH\!}om_{\dC[\theta,t,t^{-1}]}(G_0(*D),\dC[\theta,t,t^{-1}])\cong
\left\{
l\in G(*D)^\vee\,|\, l(G_0(*D)) \subset \dC[\theta,t,t^{-1}]
\right\}\subset G(*D)^\vee.
$$
In order to show this statement, we will consider completions along $\theta=0$. We write $\cO^\wedge:=\dC[[\theta]][\theta^{-1},t,t^{-1}]$.
For any $\cO$-module $\cF$ we denote by $\cF^\wedge$ the tensor product with $\cO^\wedge$.
We have thus a duality isomorphism $\Phi^{mer,\wedge}:G(*D)^\wedge \longrightarrow \left(\iota^*G(*D)^\vee\right)^\wedge$.
As both $\iota^*\left(\theta^{n-1}\cdot G_0(*D)^\vee\right)$ and $\Phi^{mer}(G_0(*D))$ are lattices
inside $\iota^*\left(G(*D)^\vee\right)$,
it is sufficient to show that $\Phi^{mer,\wedge}:G_0(*D)^\wedge \stackrel{\cong}{\longrightarrow}\left(\iota^*(\theta^{n-1}\cdot G_0(*D)^\vee)\right)^\wedge$.
In order to show this property, we will consider the formal versions of the exact sequences \eqref{eq:Resolution} and \eqref{eq:SaitoRes}, which are
exact sequences of $\cD^\wedge:=\dC[[\theta]][\theta^{-1},t,t^{-1}]\langle\partial_\theta,\partial_t\rangle$-modules.

We have thus the following formal version of the diagram \eqref{diag:HolMerDual}
{\small
\begin{equation}\label{diag:HolMerDualFormal}
\xymatrix{
\cD^\wedge \ar@^{(->}[rr]^{\left(\begin{array}{c}P^t_2-n\theta\\P^t_1\end{array}\right)\cdot} \ar[dd]_{\psi^\wedge}  && (\cD^\wedge)^2 \ar[rr]^{\D\left(P^t_1\;\;-P^t_2\right)\cdot} \ar[dd] && \cD^\wedge \ar@{->>}[r] \ar[dd]& (\cG^\wedge)^r=\cD^\wedge/(P^t_1,P^t_2)\cD^\wedge \ar[dd]_{(\varphi^\wedge)^r}^{\cong} \\ \\
(G(*D)\otimes \cD)^\wedge\ar@^{(->}[rr]^{\alpha^\wedge}         && (\Omega^1\otimes G(*D)\otimes \cD)^\wedge\ar[rr]^{\beta^\wedge}         && (\Omega^2\otimes G(*D)\otimes \cD)^\wedge \ar@{->>}[r] & (\Omega^2\otimes G(*D))^\wedge.
}
\end{equation}
}

Consider
as in \cite[lemma A.12]{DS}
the filtration $F_\bullet$ on $\cD^\wedge$ which extends the order filtration on $\dC[t,t^{-1}]\langle\partial_t\rangle$
and for which $\partial_\theta$ has degree two and $\theta$ has degree $-1$ (this is not the usual
filtration by order as considered in the proof of proposition \ref{prop:HolDual-Alg}, 1. above).
Notice that by an argument like in \cite[page 4]{SaitoSturmTaka}, we see that $gr_\bullet^F \cD^\wedge\cong \dC[[\theta]][\theta^{-1},t,t^{-1},u,v]$, where $u$ resp. $v$ represents the class of $\partial_\theta$ resp. $\partial_t$. The ring $gr_\bullet^F$ is graded by $\deg(t)=0$, $\deg(\theta)=-1$, $\deg(u)=2$ and $\deg(v)=1$.
Notice further that $F_\bullet$ induces
a good filtration on $\cG^\wedge$ with $\cG_0^\wedge=F_0\cG^\wedge$, and that we have $F_k\cG^\wedge = \theta^{-k} \cG_0^\wedge$, this follows from the fact that
$\cG^\wedge=\oplus_{i=0}^{n-1} \dC[[\theta]][\theta^{-1},t,t^{-1}] Q_i$ and that any $Q_i$ is of degree zero (i.e. the minimal $k$ such that $Q_i\in F_k \cD^\wedge$ is zero). Moreover, put $F_k G(*D)^\wedge := \theta^{-k}G_0(*D)^\wedge$, then as $F_k\cG^\wedge = \theta^{-k} \cG^\wedge_0$ the isomorphism $\varphi^\wedge:\cG^\wedge\rightarrow G(*D)^\wedge$ induced from the isomorphism $\varphi$ from theorem \ref{theo:GMSystemCyclic} is strictly filtered.

We will consider induced filtrations $F_\bullet (\Omega^1)^\wedge$ resp. $F_\bullet(\Omega^2)^\wedge$ defined by the filtration on $\cO^\wedge$ induced from
the filtration on $\cD^\wedge$. These filtrations are defined such that
$d\theta\in F_{-2}(\Omega^1)^\wedge\backslash F_{-3}(\Omega^1)^\wedge$, $dt\in F_{-1}(\Omega^1)^\wedge\backslash F_{-2}(\Omega^1)^\wedge$
and $d\theta\wedge dt\in F_{-3}(\Omega^2)^\wedge\backslash F_{-4}(\Omega^2)^\wedge$. The map of right $\cD^\wedge$-modules
$(\varphi^\wedge)^r:(\cG^\wedge)^r\rightarrow ((\Omega^2)\otimes G(*D))^\wedge$ associated to the morphism $\varphi^\wedge$
sends $1$ to $(d\theta\wedge dt)\otimes (\vol/dh)$ and
thus satisfies $(\varphi^\wedge)^r:F_\bullet (\cG^\wedge)^r \stackrel{\cong}{\rightarrow} F_{\bullet-3}\left(\Omega^2\otimes G(*D)\right)^\wedge$.

Notice that we have $\nabla_\theta F_kG(*D)^\wedge \subset F_{k+2} G(*D)^\wedge$ (because $G_0(*D)$ is stable by $\theta^2\nabla_\theta$) and
that similarly the inclusion $\nabla_t F_k G(*D)^\wedge \subset F_{k+1} G(*D)^\wedge$ holds.
Using our convention for the induced filtration on $\Omega^1$ and $\Omega^2$, this shows that the formal version of the
resolution \eqref{eq:SaitoRes} is filtered.  The same is obviously true for the formal version of the sequence \eqref{eq:Resolution},
as all components of the matrices defining the differentials in that sequence have degree $0$ for the filtration $F_\bullet\cD^\wedge$, so that we can define the filtration on each term in the standard way. Hence all horizontal maps of the diagram \eqref{diag:HolMerDualFormal} respect the induced filtration on each term.
Actually, we can show more: These sequences are even \emph{strict} resolutions of $(\cG^\wedge)^r$ resp. $(\Omega^2\otimes G(*D))^\wedge$.
Recall that a filtered complex $(\cK^\bullet, F_\bullet, d)$ of coherent $\cD$-modules is a strict
resolution of a coherent $\cD$-module $\cM$ iff for any $k\in\dZ$ the induced morphism
$(F_k\cK^\bullet,d) \rightarrow F_k\cM$ (the latter object seen as a complex concentrated in one degree) is a quasi-isomorphism.
In order to check this property, it suffices to show that the induced morphism $(gr^F_k \cK^\bullet,d) \rightarrow gr^F_k\cM$ is a quasi-isomorphism
provided that for any $i$, $F_k \cK^i$ is a finite $F_0\cD^\wedge$-module.
This criterion for strictness follows from \cite[proposition 1.1.3 d)]{SchapiraBook} (and goes back to \cite[proposition 3.2.7]{SKK}), namely, the filtration $F_\bullet\cD^\wedge$ is in fact \emph{Zariskian} in the sense of \cite[definition 1.1.2 2)]{SchapiraBook}. To show this, we remark that
the ring $\cD^\wedge$ can be identified with the ring of formal micro-differential operators
on $S\times T^*$ (more precisely, consider the sheaf of formal micro-differential operators on $T^*(S\times T^*)\backslash T_{S\times T^*}^*(S\times T^*)$, and restrict it to the image of the section $ds:S\times T^* \rightarrow T^*(S\times T^*)$ so that it can be considered as a sheaf on $S\times T^*$ and take its global sections) via Fourier-Laplace transformation sending $\theta$ to $\partial_s^{-1}$ and $\partial_\theta$ to $-s$, and then the filtration $F_\bullet$ on $\cD^\wedge$ is nothing but the filtration induced by the degree of (micro-)differential operators.
Then the Zariskian property is shown in \cite[proposition 2.2.1]{SchapiraBook}.
Notice that the finiteness over $F_0\cD^\bullet$ of each filtration step of any module in both of the horizontal exact sequences in diagram \eqref{diag:HolMerDualFormal} obviously holds.


Let us first show that the graded object of the upper sequence in diagram \eqref{diag:HolMerDualFormal} is acyclic.
This graded complex is the Koszul complex of the symbols (with respect to $F_\bullet \cD^\wedge$) of $P_1^t$ and $-P_2^t$
in $\gr_\bullet^F(\cD^\wedge)$, i.e., the Koszul complex of $(-\theta\cdot t\cdot v)^n-\frac{c}{n^n}t$ and $\theta^2\cdot u+nt\theta\cdot v$. The ideal
generated by these two functions has codimension two,
notice that
$$
\frac{\gr_\bullet^F(\cD^\wedge)}{((-\theta\cdot t\cdot v)^n-\frac{c}{n^n}t,\theta^2\cdot u+nt\theta\cdot v)}
=
\frac{\gr_\bullet^F(\cD^\wedge)}{((-\theta\cdot t\cdot v)^n-\frac{c}{n^n}t,u+\theta^{-1}\cdot nt\cdot v)}
\cong
\frac{\dC[[\theta]][\theta^{-1},t,t^{-1},v]}{((-\theta\cdot t\cdot v)^n-\frac{c}{n^n}t)}
$$
and the latter ring is obviously two-dimensional. Hence the symbols of $P_1^t$ and $-P_2^t$ define a complete intersection and thus form a regular sequence
in $\gr_\bullet^F(\cD^\wedge)$. It follows that the Koszul complex of these two functions is a resolution of the quotient ring. We conclude that the upper line of diagram \eqref{diag:HolMerDualFormal} is a strict resolution of $(\cG^\wedge)^r$.

A similar argument applies to the lower line of this diagram: The graded complex
$\gr_\bullet^F(\Omega^\bullet(\cD^\wedge))$ is a resolution of $gr_\bullet^F((\Omega^2)^\wedge)$ since it is simply the Koszul complex of the elements $u,v$ in the ring $\gr^F_\bullet(\cD^\wedge)=\dC[[\theta]][\theta^{-1},t,t^{-1},u,v]$. Hence the graded complex of $(\Omega^\bullet(\cD)\otimes G(*D))^\wedge$ is a resolution of $\gr^F_\bullet(\Omega^2\otimes G(*D)^\wedge)$ (since the differential on $(\Omega^\bullet(\cD)\otimes G(*D))^\wedge$ is the identity on the second factor), so that $(\Omega^\bullet(\cD)\otimes G(*D))^\wedge$ is a strictly filtered resolution of $(\Omega^2\otimes G(*D))^\wedge$.
Moreover, the isomorphism $((\Omega^\bullet\otimes
\cD)\otimes G(*D))^\wedge \stackrel{\cong}{\rightarrow}((\Omega^\bullet\otimes G(*D))\otimes \cD)^\wedge$ used in the proof of theorem \ref{theo:Pairing} is strict because
it is filtered and its inverse (induced by the isomorphism $(G(*D)\otimes \cD)^\wedge \stackrel{\cong}{\rightarrow} (\cD\otimes G(*D))^\wedge$) also respects the filtration. We conclude that the formal version of the complex \eqref{eq:SaitoRes}, i.e., the lower line of diagram \eqref{diag:HolMerDualFormal} is a strict
resolution of $(\Omega^2 \otimes G(*D))^\wedge$.

We have now seen that both lines of this diagram are strict, and moreover that the rightmost isomorphism $(\varphi^\wedge)^r$ strictly shifts
the filtration by $-3$. Then the same holds for any of the vertical morphisms, and the induced isomorphism
$\Psi^\wedge$ satisfies
$$
\Psi^\wedge:\left(F^\vee_{\bullet+3} (G(*D)^\vee) \right)^\wedge \stackrel{\cong}{\longrightarrow} \left(F^{\bD}_\bullet (\bD \cG)\right)^\wedge.
$$
where
$F^\vee_\bullet (G(*D)^\vee) := {\cH\!}om_{\dC[\theta,t,t^{-1}]}(F_\bullet G(*D), \dC[\theta,t,t^{-1}])$
and where $F^{\bD}_\bullet(\bD \cG)$ is the filtration on $\bD \cG$ dual to $F_\bullet \cG$ in the sense of
\cite[§ 2.4]{Saito1}. Moreover, due to the strictness of the resolution of $(\cG^\wedge)^r$, the (formal version of the)
duality isomorphism from proposition \ref{prop:HolDual-Alg} satisfies
\begin{equation}\label{eq:FilteredDual}
\Phi^\wedge:(\cG^\wedge,F^\wedge_\bullet) \stackrel{\cong}{\longrightarrow} \iota^*((\bD\cG)^\wedge,(F^{\bD}_{\bullet-(n+2)})^\wedge),
\end{equation}
where again $F^{\bD}_\bullet$ is the
filtration dual to $F_\bullet$. Combining this with the isomorphism $\varphi^\wedge:(F_0\cG)^\wedge=\cG^\wedge_0\stackrel{\cong}{\rightarrow}G_0(*D)^\wedge$,
noticing that $F^\vee_0 (G(*D)^\vee)=G_0(*D)^\vee$  and
looking at the (formal version of the) diagram \eqref{diag:PhiInduced} we finally conclude that
$$
\Phi^{mer,\wedge}=(\iota^*\Psi)^\wedge\circ \Phi^\wedge \circ (\varphi^{-1})^\wedge:G_0(*D)^\wedge \stackrel{\cong}{\longrightarrow}\left(\iota^* (\theta^{n-1}\cdot G_0(*D)^\vee)\right)^\wedge,
$$
as required.
\end{proof}

\section{Limit and weak logarithmic Frobenius structures}
\label{sec:Frobenius}

In this final section we indicate how the theorems \ref{theo:Pairing} and \ref{theo:PairingLattice} from the last section
can be used to sharpen the results from \cite[section 4 and 5]{dGMS}. We start with a preliminary lemma.
Consider the grading of $\dC[\theta,\theta^{-1},t,t^{-1}]$ for which $\deg(\theta)=1$ and $\deg(t)=n$.
We obtain an induced grading on $\cD$ with $\deg(\partial_\theta)=-1$
and $\deg(\partial_t)=-n$. Notice that this grading does neither induce the filtration $F_\bullet$ on $\cD$ considered in
the proof of theorem \ref{theo:PairingLattice} nor the usual order filtration used in the proof of proposition \ref{prop:HolDual-Alg}.
The module $G(*D)$ (as well as its various lattices) carries a compatible grading, we have $\deg(\omega_i)=i-1$ and
the connection operator $\nabla$ on $G(*D)$ from proposition \ref{prop:GoodBasis} as well as the isomorphism
$\varphi$ from theorem \ref{theo:GMSystemCyclic} are homogenous of degree $0$.
\begin{lemma}{(see also \cite[conjecture 5.5]{dGMS})}\label{lem:Homog}
The pairing $S$ from theorem \ref{theo:Pairing} is homogenous of degree $0$ with respect to this natural grading of $(G(*D),\nabla)$,
i.e., we have
$$
S\left(G(*D)_k,G(*D)_l\right)\subset \dC[\theta,\theta^{-1},t,t^{-1}]_{k+l}.
$$
\end{lemma}
\begin{proof}
The statement we need is equivalent to the fact that the isomorphism $\Phi^{\mathit{mer}}$ from the proof of theorem \ref{theo:Pairing} is
homogenous of degree $0$. In order to show this, we need again to consider the comparison
isomorphism
$\Psi:(G(*D),\nabla)^\vee \rightarrow \bD(G(*D))$ from the proof of theorem \ref{theo:Pairing}.
Recall that it was constructed from diagram \eqref{diag:HolMerDual}, which we recall below. However, we will write
it in such a way that all morphism are homogenous for the above grading. We have
{\small
$$
\xymatrix{
\cD[n+1] \ar@^{(->}[rr]^{\left(\begin{array}{c}P^t_2-n\theta\\P^t_1\end{array}\right)} \ar[dd]_{\psi'}  && \cD[n]\oplus\cD[1]\ar[rr]^{\D\left(P^t_1\;\;-P^t_2\right)\cdot} \ar[dd] && \cD \ar@{->>}[r] \ar[dd]& \cG^r=\cD/(P^t_1,P^t_2)\cD \ar[dd]_{\varphi^r}^{\cong} \\ \\
G(*D)\otimes_\cO \cD \ar@^{(->}[rr]^\alpha         && \Omega^1\otimes_\cO G(*D)\otimes_\cO \cD \ar[rr]^\beta         && \Omega^2\otimes_\cO G(*D)\otimes_\cO \cD \ar@{->>}[r] & \Omega^2\otimes_\cO G(*D)
}
$$
}
$\!\!\!$where $\cD[k]$ is $\cD$ with the shifted grading defined by $\cD[k]_l=\cD_{k-l}$.
From the description of the morphisms $\alpha$ and $\beta$ given in the proof of theorem \ref{theo:Pairing} we see
that both are homogenous of degree $0$.
Hence the above diagram is a morphism of $0$-graded complexes of right $\cD$-modules
(i.e., a morphism of complexes of graded right $\cD$-modules with differentials of degree $0$).
Therefore the degrees of the vertical maps are all equal, and we only have to determine the degree of the rightmost morphism
$\varphi^r$. Recall that it is defined by sending $1$ to $(d\theta\times dt)\otimes(n\vol/dh)$, hence, $\deg(\varphi^r)=n+1$.
It follows that $\deg(\psi')=n+1$. The morphism $\psi:\cD\rightarrow G(*D)\otimes_{\cO}\cD$ from diagram \eqref{diag:HolMerDual}
is the same as $\psi'$, but defined
on $\cD$ rather than on  $\cD[n+1]$, hence $\deg(\psi)=2n+2$. It follows that we have $\deg(\Psi)=2n+2$, where
$\Psi:G(*D)^\vee \rightarrow \bD\cG$ is the isomorphism induced by the dual of $\psi$.
Now consider again diagram \eqref{diag:PhiInduced} from above.
From $\deg(\varphi)=0$ and $\deg(\Phi)=\deg(\iota^*\Psi)=2n+2$ we deduce $\deg(\Phi^{\mathit{mer}})=0$, as required.
\end{proof}

As a consequence, we obtain a quite precise result on the behavior of the pairing $S$ with respect
to the bases $\underline{\omega}$ and $\underline{\widetilde{\omega}}$ from proposition \ref{prop:GoodBasis}.
\begin{theorem}
\label{theo:VSSolution}
Let $D\subset V$ be reductive and $f\in V^\vee$ generic. Then
\begin{enumerate}
\item
The basis $\underline{\widetilde{\omega}}$ from proposition \ref{prop:GoodBasis}
yields a $(V^+,S)$-solution (in the sense of \cite[appendix B.d]{DS}) to the Birkhoff problem in family of $(G_0(*D),\nabla)$ and more precisely, we have
\begin{equation}\label{eq:Sconstant-1}
S(\widetilde{\omega}_i,\widetilde{\omega}_j) \in \dC \cdot \theta^{n-1}\cdot t^{2k}\cdot\delta_{i+j,n+1}
\end{equation}
for some $k\in\dN$.
\item
The basis $\underline{\omega}$ from proposition \ref{prop:GoodBasis}
is a $(V^+,S)$-solution to the Birkhoff problem of $(G(\log\,D)/t\cdot G(\log\,D),\nabla)$ and more precisely, we have
\begin{equation}\label{eq:Sconstant-2}
S(\omega_i,\omega_j) \in \dC \cdot \theta^{n-1}\cdot \delta_{i+j,n+1}
\end{equation}
\item
We have $S(G_0(\log\,D),G_0(\log\,D))\subset \theta^{n-1}\dC[\theta,t]$ (this is the second part of \cite[conjecture 5.5.]{dGMS}).
\item
Put $\widehat{G}(*D):=\oplus_{i=1}^n \cO_{\dP^1\times T^*} \cdot \widetilde{\omega}_i$  and
$\widehat{G}(\log\,D):=\oplus_{i=1}^n \cO_{\dP^1\times T} \cdot \omega_i$, , then $S$ induces non-degenerate pairings
$$
S:\widehat{G}(*D) \otimes \iota^*\widehat{G}(*D) \rightarrow \cO_{\dP^1\times T^*}(-(n-1),n-1)
$$
and
$$
S:\left(\widehat{G}(\log\,D)/t\cdot\widehat{G}(\log\,D)\right) \otimes \iota^*\left(\widehat{G}(\log\,D)/t\cdot \widehat{G}(\log\,D)\right) \rightarrow \cO_{\dP^1}(-(n-1),n-1).
$$
where $\cO_{\dP^1\times T^*}(a,b)$ (resp. $\cO_{\dP^1}(a,b)$) is the subsheaf of $\cO_{\dP^1\times T^*}(*\{0,\infty\}\times T^*)$ (resp.
$\cO_{\dP^1}(*\{0,\infty\})$) of meromorphic functions with a pole of order $a$ along $\{0\}\times T^*$ (resp. at $0$) and a pole of order $b$ along $\{\infty\}\times T^*$ (resp. at $\infty$).
\end{enumerate}
\end{theorem}
\begin{proof}
It was shown in \cite[corollary 4.12]{dGMS} that $\underline{\widetilde{\omega}}$ (resp. $\underline{\omega}$) is a $V^+$-solution
for $(G_0(*D),\nabla)$ (resp. for $(G_0(\log\,D)/t\cdot G_0(\log\,D),\nabla)$). Recall from
\cite[appendix B.d]{DS} (see also \cite[lemma-definition 4.6]{dGMS}) that $\underline{\widetilde{\omega}}$ (resp. $\underline{\omega}$) is compatible
with $S$ (i.e., a $S$-solution) iff $S(\widetilde{\omega}_i,\widetilde{\omega}_j)\in\theta^{n-1}\cdot\dC[t,t^{-1}]$ for all $i\in\{1,\ldots,n\}$
(resp. $S(\omega_i,\omega_j)\in\theta^{n-1}\cdot\dC$ for all $i\in\{1,\ldots,n\}$).

%

In order to show the statements 1. and 2., we use the proof of \cite[theorem 4.13]{dGMS}, where an additional hypothesis on the
multiplicity of the smallest spectral number $\nu_{\sigma(1)}$ was made. However, as $\deg(\omega_i)=i-1$ (resp. $\deg(\widetilde{\omega})=i-1+k\cdot n$ for some $k\in\dN$)
 we deduce from the theorems \ref{theo:Pairing} and \ref{theo:PairingLattice} as well as from lemma \ref{lem:Homog} that
whenever we take $i\in\{1,\ldots,n\}$ such that $\nu_i=\nu_{\sigma(1)}$, then $S(\omega_i,\omega_k) \in \dC\cdot\delta_{i+j,n+1}\cdot\theta^{n-1}$
and $S(\widetilde{\omega}_i,\widetilde{\omega}_k)\in
\dC \cdot \theta^{n-1}\cdot t^{2k}\cdot\delta_{i+j,n+1}$.
Then the proof of theorem 4.13 in loc.cit. shows that this is true for any $i\in\{1,\ldots,n\}$,
that is, we obtain the formulas \eqref{eq:Sconstant-1} and \eqref{eq:Sconstant-2}, but also
the statements 3. and 4 from above.
\end{proof}

We can now give the promised application of the above results. It consist of a sharpening of
theorems 5.1, 5.7 and 5.9 from \cite{dGMS}. We we do not need anymore to make an assumption
on the multiplicity of the spectral number at infinity of $f_{|h^{-1}(t)}$ and
we also know that \cite[conjecture 5.5.]{dGMS} holds. Using this, the proofs of the theorems
below are the same as in loc.cit. and are therefore omitted.
We always suppose that $D$ is reductive and $f\in V^\vee$ generic.

The first result is the construction of a Frobenius manifold structure
on a miniversal deformation space of the restriction $f_{|h^{-1}(t)}$.
\begin{theorem}{\cite[theorem 5.1]{dGMS}}
Consider a semi-universal unfolding $F:B_t\times M_t\rightarrow D_\delta$
as in loc.cit. theorem 5.1 (where $D_\delta\subset \dC$ is a small disc and $B_t=h^{-1}(t)\cap B_\epsilon$, with $M_t$ and $B_\epsilon$ being balls
in $\dC^n$). Then (any non-zero multiple of) any of the section $\omega_i$ from proposition \ref{prop:GoodBasis} is primitive and homogenous and induces a Frobenius
structure on $M_t$.
\end{theorem}

The next result gives the construction of a Frobenius structure at $t=0$, that one may
see as associated to the restriction of $f$ to $D$.
\begin{theorem}{(see also \cite[theorem 5.7]{dGMS})}
\label{theo:LimitFrobenius}
The germ at the origin of the $\cR_h$-miniversal deformation space of $f$ (called $M_0$) carries a constant
Frobenius structure (i.e., such that the structure constant of the multiplication are constant in flat coordinates).
\end{theorem}

Finally, we also obtain the structure of a ''weak logarithmic Frobenius manifold'' associated to the couple $(f,h)$
(see \cite{Reich1} for the definition of a logarithmic Frobenius manifold and \cite[definition 5.8]{dGMS} for a variant
called ''weak logarithmic'').
\begin{theorem}{(see also \cite[theorem 5.9]{dGMS})}
\label{theo:WeakLogFrob}
The module $G'$ from loc.cit., theorem 5.9 underly a weak log $\Sigma$-trTLEP-structure and the form $t^{-k}\widetilde{\omega}_i=t^{-k}\omega_i^{(3)}$ (where $i$ is
the index from loc.cit., lemma 5.3) is primitive and homogenous. It yields a weak logarithmic Frobenius manifold associated to $(f,h)$.
\end{theorem}

\bibliographystyle{amsalpha}
\def\cprime{$'$} \def\cprime{$'$}
\providecommand{\bysame}{\leavevmode\hbox to3em{\hrulefill}\thinspace}
\providecommand{\MR}{\relax\ifhmode\unskip\space\fi MR }
\providecommand{\MRhref}[2]{%
  \href{http://www.ams.org/mathscinet-getitem?mr=#1}{#2}
}
\providecommand{\href}[2]{#2}

\vspace*{1cm}

\nd
Lehrstuhl f\"ur Mathematik VI \\
Institut f\"ur Mathematik\\
Universit\"at Mannheim,
A 5, 6 \\
68131 Mannheim\\
Germany

\vspace*{1cm}

\nd
Christian.Sevenheck@math.uni-mannheim.de

\end{document}